\documentclass[12pt,a4paper]{amsart}

\usepackage{fullpage}
\usepackage{etex}
\usepackage{amssymb}
\usepackage{amsmath}
\usepackage{amscd}
\usepackage[latin1]{inputenc}
\usepackage{wrapfig}
\usepackage{psfrag}
\usepackage{epsfig}
\usepackage{epic}
\usepackage{eepic}
\usepackage{pifont}
\usepackage{dsfont}
\usepackage{bbold}

\usepackage[usenames,dvipsnames]{color}

\usepackage{hyperref}

\usepackage{caption}

\usepackage{enumitem}

\usepackage{mathrsfs}

\usepackage{subcaption}

\usepackage{rotating}
\usepackage{mathtools}

\usepackage{setspace}
\usepackage[section]{algorithm}
\usepackage{algorithmicx,algpseudocode}

\usepackage[all]{xy}

\def \R  {{\mathbb R}}

\def \Z  {{\mathbb Z}}
\def \Q  {{\mathbb Q}}
\def \C  {{\mathbb C}}
\def \N  {{\mathbb N}}

\def \CP {{{\mathbb C}{\mathbb P}}}

\def \CP {{\mathbb C}{\mathbb P}}
\def \t  {{\mathfrak t}}

\def    \PD {\text{PD}}

\newcommand{\F}{{\mathbb{F}}}

\def    \ol {\overline}

\DeclareMathOperator{\grad}{grad}

\DeclareMathOperator \pt {pt}

\newcommand{\acts}{\mathbin{\raisebox{-.5pt}{\reflectbox{\begin{sideways}$\circlearrowleft$\end{sideways}}}}}

\numberwithin{figure}{section}
\numberwithin{table}{section}
\numberwithin{equation}{section}
\swapnumbers

\makeatletter
\let\c@equation\c@figure
\makeatother

\makeatletter
\let\c@table\c@figure
\makeatother

\makeatletter
\let\c@algorithm\c@figure
\makeatother

\newtheorem{Lemma}[equation]{Lemma}
\newtheorem{Theorem}[equation]{Theorem}
\newtheorem*{thm*}{Theorem}

\newtheorem{Question*}{Question}

\newtheorem{Proposition}[equation]{Proposition}
\newtheorem{Corollary}[equation]{Corollary}

\newtheorem*{Lemma*}{Lemma}
\newtheorem*{Corollary*}{Corollary}

\theoremstyle{definition}

\newtheorem{Remark}[equation]{Remark}
\newtheorem{noTitle}[equation]{}

\begin{document}

\title{The Equivariant Cohomology\\ of Complexity One Spaces}

\author{Tara S. Holm}
\address{Department of Mathematics, Cornell University, Ithaca, NY  14853-4201, USA}
\email{tsh@math.cornell.edu}

\author{Liat Kessler}
\address{Department of Mathematics, Physics, and Computer Science, University of Haifa,
at Oranim, Tivon 36006, Israel}
\email{lkessler@math.haifa.ac.il}

\begin{abstract}
Complexity one spaces are an important class of examples in symplectic geometry.
They are less restrictive than toric symplectic manifolds.  
Delzant has established that toric symplectic manifolds
are completely determined by their moment polytope.
Danilov proved that the ordinary and equivariant cohomology rings are dictated by
the combinatorics of this polytope.
These results are not true for complexity one spaces.  
In this paper, we 
describe the equivariant cohomology for a
Hamiltonian $S^1\acts M^4$. 
We then assemble the equivariant cohomology of a complexity one space from the equivariant cohomology of the $2-$ and $4$-dimensional pieces, as a subring of the equivariant cohomology of its fixed points.
We also show how to
compute equivariant characteristic classes in dimension four.
\end{abstract}

\maketitle

\section{Introduction}

A symplectic action of a torus $T=(S^1)^k$ on a  symplectic manifold 
$(M,\omega)$ is called {\bf Hamiltonian} if it admits a {\bf momentum map}.
That is, we have a smooth map $\Phi \colon M \to {\t}^{*}\simeq {\R}^k$ such that  $ d \Phi_j =  -\iota(\xi_j) \omega $
for every $j=1,\dots,k$, where $\xi_1,\ldots,\xi_k$ are the vector fields
that generate the torus action.
If the Hamiltonian action is effective, the triple $(M,\omega,\Phi)$ is called a {\bf Hamiltonian $T$-space}.
In this paper, we will consider closed and connected Hamiltonian $T$-spaces.
In a Hamiltonian $T$-space, the vector fields $\xi_1,\ldots,\xi_k$ 
define an isotropic subbundle of the tangent bundle, so we must have
$\dim(T)\leq \frac{1}{2}\dim(M).$
When we have the equality $\dim(T)= \frac{1}{2}\dim(M)$, 
the action is called {\bf toric}, and the space is called a \textbf{toric symplectic manifold}.  Delzant has shown that closed connected toric symplectic
manifolds, up to equivariant symplectomorphism, 
are in one-to-one correspondence with simple, rational, smooth polytopes,
up to affine equivalence.

 More generally, an effective Hamiltonian action
$(S^1)^{n-k}\acts M^{2n}$  has {\bf complexity $k$}.  Hence,
a {\bf complexity one space} is a symplectic manifold equipped with an effective
Hamiltonian action of a torus which is one dimension
less than half the dimension of the manifold.  
The only 
example of an effective Hamiltonian action on a two-dimensional closed connected $M$ is
a linear action $S^1\acts S^2$, which is toric.
When $M$ is four-dimensional, the only tori that act effectively are $S^1$ 
and $S^1\times S^1$. Thus examples of Hamiltonian circle actions $S^1\acts M^4$ 
are the first examples of complexity one spaces.

Already for $S^1\acts M^4$, we see that an analogue of Delzant's theorem is impossible.
Delzant's theorem says that $T^2\acts M^4$ correspond to a class of polygons in $\R^2$.  
Restricting to a circle, we have moment image a single interval: we've lost too much information to be
able to recover any topological information about $M$. Nevertheless, the existence of an effective Hamiltonian circle
action on a compact connected symplectic $4$-manifold does have topological implications.  Li has shown \cite{Li:2006} that for
an effective Hamiltonian circle action, the fundamental group $\pi_1(M)$ must be isomorphic to  the fundamental group
$\pi_1(\Sigma_{min})$ of the submanifold $\Sigma_{min}$ on which the moment map $\Phi$ achieves its minimum.
This minimum is automatically a symplectic submanifold, so it must be an isolated point or an oriented surface.
On the other hand, thanks to a construction of Gompf \cite{Gompf:1994}, any finitely presented group can be the 
fundamental group of a compact symplectic $4$-manifold.  Thus, the existence of an effective Hamiltonian circle
action restricts the topology of $M$ considerably.

In general, Hamiltonian $T$-spaces enjoy a number of useful features when it comes to computational
topology.  Components of the momentum map are Morse functions (in the sense of Bott).  Thus, topological
invariants like singular cohomology are amenable to computation for these spaces \cite{frankel:1959}.
More subtly, it is often possible to compute {\bf equivariant cohomology}, an invariant depending on both the
manifold and the action. 
As the critical set for a generic component of the momentum map, the set $M^T$ of fixed points plays a leading role
in these calculations.    
For an effective Hamiltonian action $S^1\acts M^4$ with only isolated fixed points, Goldin and the
first author \cite{Goldin-Holm:2001} use the Atiyah-Bott/Berline-Vergne (ABBV) localization formula  \cite{AB:1984, BV:1982} to
describe 
the equivariant cohomology $H_{S^1}^*(M;\Q)$.
In this case, the $S^1$-action extends to a toric
action $T^2\acts M^4$. 
In general, an effective Hamiltonian $S^1$-action on a four-manifold
might fix two-dimensional submanifolds and it need not extend to a toric action.
We 
review the relevant facts on equivariant cohomology in Section~\ref{se:eq coh}.

The first main result of this manuscript describes the $S^1$-equivariant cohomology
for any effective Hamiltonian $S^1$-action on a symplectic four-manifold. 
Examples include $k$-fold blowups of symplectic ruled surfaces of positive genus. 
This is a rare instance
in the symplectic category where the presence of odd degree cohomology doesn't
make calculations in equivariant cohomology impossible.  
It is also the first occurrence of calculations with
fixed point components of different diffeomorphism types.

\begin{Theorem}\label{thm:S1onM4}
Let $M$ be a closed connected symplectic four-manifold, and $S^1\acts M$ be an effective Hamiltonian circle action.
\begin{enumerate}
\item[(A)]
The equivariant cohomology $H_{S^1}^*(M;\Z)$ is a free $H_{S^1}^*(pt;\Z)$-module
isomorphic to $$H^*(M;\Z)\otimes H_{S^1}^*(pt;\Z).$$
More precisely,  if $F_1,\dots, F_s$ are the (finitely many) components of the fixed set $M^{S^1}$, 
then there are even natural numbers $\lambda_1,\dots,\lambda_s$ such that
$$
H_{S^1}^*(M;\Z) = \bigoplus_j H^{*-\lambda_j}(F_j;\Z)\otimes H_{S^1}^*(pt;\Z).
$$

\item[(B)]
The inclusion $i:M^{S^1}\hookrightarrow M$ induces an injection in 
integral equivariant cohomology
$$
i^*:H^*_{S^1}(M;\Z) \hookrightarrow H^*_{S^1}\left(M^{S^1};\Z\right).
$$

\item[(C)]
In equivariant cohomology with rational coefficients, the image of $i^*$ is characterized
as those classes  $\displaystyle{\alpha\in H^*_{S^1}\left(M^{S^1};\Q\right) 
= \bigoplus_{F\subset M^{S^1}}H^*_{S^1}\Big(F;\Q\Big) }$ which satisfy:
\begin{enumerate}
\item[(0)] that the degree zero terms $\alpha^{(0)}|_F$ are all equal;
\item[(1)] that the degree one terms $\alpha^{(1)}|_\Sigma$ restricted to fixed surfaces
are equal; and
\item[(2)] the ABBV relation
\begin{equation}\label{eq:relations}
 \sum_{{F\subset M^{S^1}}} \left(\pi|_F\right)^{!}\left( \frac{ \alpha|_F}{e_{S^1}(\nu(F\subseteq M))}\right) \in \Q[u]=H_{S^1}^*(\pt;\Q),
\end{equation}
where the sum is taken over the connected components $F$ of the fixed point set $M^{S^1}$, 
 $\alpha|_F$ is the restriction of $\alpha$ to the component $F$,  the map $\left({\pi|_F}\right)^{!}$ is the
equivariant pushforward of $\pi|_{F} \colon F \to {\pt}$, and  $e_{S^1}(\nu(F\subseteq M))$ is the equivariant Euler class of the normal bundle of $F$.

\end{enumerate}
\item[(D)] The Atiyah-Bredon sequence
for $S^1\acts M$
is exact over $\Z$.
\item[(E)]
We have a short exact sequence
$$0 \to \langle \pi^{*}(u) \rangle \to  H_{S^1}^{*}(M;\Z) \to H^{*}(M;\Z) \to 0.$$
\end{enumerate}
\end{Theorem}

For Parts (A) and (B), 
we adapt Morse theoretic arguments of Tolman and Weitsman
 \cite{TW-hamTsp:1999} over $\Q$ to show they work for our class of manifolds when the coefficients are the integers.
 Franz and Puppe have also adapated this argument to work over $\Z$  (see  \cite[Proof of Theorem 5.1]{Franz-Puppe:2007}),
 but they require the action to have connected stabilizer groups, which we do not have.
 By contrast, we do have torsion-free fixed point components, which allows us to draw the same conclusion.
 We then use Franz and Puppe's work to verify that (D) and (E) hold for these 4-manifolds.

Turning to part (C) of the main theorem,
we exhibit a sample class satisfying the requirements listed in (C) in Figure~\ref{fig:sample class}.
We prove this part using
Morse theory to compute the equivariant Poincar\'e polynomials
$P_{S^1}^{M^{S^1}}(t)$ and $P_{S^1}^M(t)$ and their difference. 
This tells us the ranks of 
$i^{*}(H_{S^1}^*(M;\Q))$.

To interpret the ABBV relation, we calculate explicitly equivariant Euler classes and their inverses. 
In Section \ref{se:new} we give formulas for 
these classes, 
and for equivariant 
Chern classes, in terms of the weights of the action at the fixed points and the self intersection of the fixed 
surfaces. 
In the appendix we apply the ABBV relation and our calculations to compute the intersection numbers of embedded invariant surfaces.

\begin{figure}[h]
\centering
\includegraphics[height=4cm]{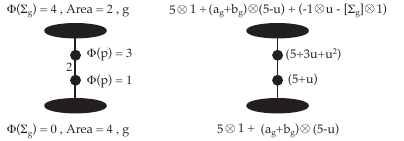} 
\caption[.]{{
On the left  is the decorated graph for a circle action on the manifold 
$M=[\C P^2_{\mathbf 4} \times (\Sigma_g)_{\mathbf{4}}]\# \ol{{\C P}}^2_{\mathbf{2}}\# \ol{{\C P}}^2_{\mathbf{1}}$
and on the right a collection
of classes in $H^*_{S^1}(F)$ for each fixed component $F$.
These classes satisfy the requirements in Theorem~\ref{thm:S1onM4},
so they are the restrictions to the fixed sets of a global class in
$H_{S^1}^*(M;\Q)$. 
}}
\label{fig:sample class}
\end{figure}

In the second main result of the manuscript, Corollary~\ref{thm:complexity one}, we  generalize Theorem \ref{thm:S1onM4}, applying 
a theorem of Tolman and Weitsman \cite{TW-hamTsp:1999} to assemble 
the equivariant cohomology of a complexity one space, 
$T^{n-1}\acts M^{2n}$.
Tolman and Weitsman's work is a consequence of an earlier result of Chang and Skjelbred \cite{CS:1974} but the
Tolman-Weitsman proof illuminates the type of geometric argument that we use in the Hamiltonian setting.
Our main results demonstrate how amenable complexity one spaces
are to algebraic computation.
It opens the door to questions about the geometric data encoded
in the equivariant cohomology ring for complexity one spaces,
along the lines of Masuda's work \cite{Masuda:2008}, distinguishing toric manifolds by their equivariant cohomology.
 
 \bigskip

\noindent {\bf Acknowledgements.}
We would like to thank Yael Karshon for helpful conversations about complexity
one spaces, and we are grateful for the hospitality of the Bar Natan Karshon Hostel
during the completion of our manuscript.  The first author was supported in part
by the National Science Foundation under Grant DMS--1206466. Any opinions, findings, and conclusions or recommendations expressed in this material are those of the authors and do
not necessarily reflect the views of the National Science Foundation.

\section{Hamiltonian circle actions on $4$-manifolds}\label{se:dec graphs}

Let $(M^4,\omega)$ be a closed connected symplectic four-manifold with an effective Hamiltonian $S^1$-action.
The real-valued momentum map
$\Phi: M\to \R$ is a Morse-Bott function with
critical set corresponding to the fixed points  \cite[\S 32]{GS-book:1984}.
Since $M$ is four-dimensional, the critical set can only consist
of isolated points and two-dimensional submanifolds.  
The latter
can only occur at the extrema of $\Phi$.  The maximum and minimum of the momentum map is each attained on exactly one component of the fixed point set. This is the key point for our computations below.

To $(M,\omega,\Phi)$ Karshon associates a {\bf decorated graph}. The decorated graphs
for two different circle actions on  $\C P^2$ are shown in Figure~\ref{fig:dec graph}.
\begin{figure}[h]
\centering
\includegraphics[width=10cm]{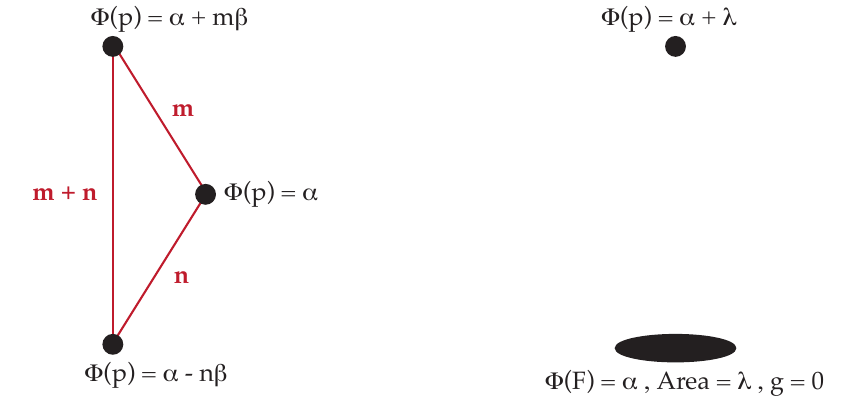} 
\caption[.]{The decorated graphs for $S^1\acts \C P^2$: with $3$ isolated fixed points 
on the left; and with
a fixed surface and an isolated fixed point on the right. An edge labelled $\ell$ indicates a $\Z_\ell$-sphere containing the two fixed points.}
\label{fig:dec graph}
 \end{figure}
  Two Hamiltonian $S^1$-spaces are isomorphic if and only if they have the same decorated graph \cite{Karshon:1999}. 
  Moreover, we know which such spaces occur \cite{Karshon:1999}.
 In particular, when the fixed points of the action $S^1\acts M^4$ are isolated, the $S^1$-action extends to a toric
action $T^2\acts M^4$.  If there is a single critical surface $\Sigma$, then we may deduce that $\Sigma$ has genus $0$.
When there are two fixed surfaces $\Sigma_{min}$ and $\Sigma_{max}$, 
they must have the same genus, so are each homeomorphic
to a fixed surface $\Sigma$.   
We call the case when there are two fixed surfaces of genus $g>0$ 
the {\bf positive genus} case.

\section{Equivariant cohomology}\label{se:eq coh}

In this paper, we are interested in computing the (Borel) equivariant 
cohomology of a complexity one space.
Recall that for a circle action on a manifold $M$
$$
H_{S^1}^*(M;R) := H^*((M\times ES^1)/S^1; R),
$$
where the classifying bundle 
$ES^1 := S^{\infty}$ is the unit sphere in an infinite
dimensional complex Hilbert space $\C^\infty$,  equipped with a free $S^1$-action by
coordinate multiplication, 
$S^1\acts (M\times ES^1)$ diagonally, and $R$ is the 
coefficient ring.  The classifying space is $BS^1 = ES^1/S^1 = \C P^\infty$.
The equivariant cohomology of a point is
\begin{equation} \label{eqhs1pt}
H^*_{S^1}(\pt; R) = H^*(BS^1;R) = H^*(\C P^\infty;R) = R[u],
\end{equation}
where $\deg(u) =2$.

\begin{Remark} \label{finiteappro}
We can interpret $ES^1=S^{\infty}$ as the direct limit of odd-dimensional spheres $S^{2\ell+1} \subset {\C}^{\ell+1}$ with respect to the natural inclusions, and $BS^1=\CP^{\infty}= \lim \limits_{\longrightarrow} {\CP^\ell}$.
Then $(M \times ES^1) / S^1$ is a direct limit of $(M \times S^{2\ell+1})/ S^1 $. For every degree $q$ we have 
$H_{S^1}^{q}(M)=H^{q}((M \times S^{2\ell+1})/ S^1;\Z)$ for all sufficiently large $\ell$. 
\end{Remark}

More generally for a torus $T = (S^1)^k$, we have 
$
H_{T^k}^*(M;R) := H^*((M\times (ES^1)^k)/T^k; R).
$ 
The inclusion of the fixed point set $i:M^T\hookrightarrow M$ is an equivariant map, and 
Borel studied the induced map in equivariant 
cohomology, using field coefficients.
\begin{Theorem}[{Borel \cite{Borel:1960}}] \label{thm:ker coker}
Let a torus $T$ act on a closed manifold $M$ and let $\F$ be a field of characteristic zero.
In equivariant cohomology, the kernel and cokernel of the map induced by inclusion,
$$
i^*:H_T^*(M;\F)\to H_T^*(M^T;\F)
$$
are torsion submodules.  In particular, if $H_T^*(M;\F)$ is torsion free as a module over $H_T^*(\pt;\F)$, then $i^*$ is
injective.
\end{Theorem}

The map
\begin{equation} \label{eq:pi1}
\pi:M\to \pt
\end{equation} induces a map
in equivariant cohomology $\pi^*:H_T^*(\pt;R)\to H_T^*(M;R)$ which endows
$H_T^*(M;R)$ with an $H_T^*(\pt;R)$-module structure. 
The projection \eqref{eq:pi1} induces a fibration
$$(M\times E{T})/{T} \to ET/T=BT.$$
In the context of Hamiltonian torus actions, Kirwan studied 
the $H_T^*(\pt;R)$-module structure for coefficient rings 
which are fields of characteristic $0$.  
Kirwan proved the following, adapted to our context.

\begin{Theorem}[{Kirwan \cite[(5.8)]{kirwan:1984}}] \label{thm:module}
Consider a Hamiltonian action of a torus $T$ on a closed manifold $M$.
If $\F$ is a field
of characteristic zero, then $H_T^*(M;\F)$ is a free $H_T^*(\pt;\F)$-module
isomorphic to $H^*(M;\F)\otimes H_T^*(\pt;\F)$.
\end{Theorem}

\begin{noTitle}  \label{nTpush}
{\bf Pushforward maps.}
An  $S^1$-equivariant continuous map of closed oriented $S^1$-manifolds, $f \colon N \to M$  induces 
 the {\bf equivariant pushforward map}
$$f^{!} \colon H_{S^1}^{*}(N;\Z) \to H_{S^1}^{*-n+m}(M;\Z),$$
where $n=\dim N, \, m=\dim M$, as follows.
 For $q, \, \ell \in \N$,  
 we have the pushforward homomorphism 
$ H^{q}(N \times S^{2\ell+1}/S^1;\Z) \to H^{q-n+m}(M \times S^{2\ell+1}/S^1;\Z)$
defined by
$$
\xymatrix{
H^{q}(N \times S^{2\ell+1}/S^1;\Z) \ar[d]  & H^{q+m-n}(M \times S^{2\ell+1}/S^1;\Z)\\
H_{n-q}(N \times S^{2\ell+1}/S^1;\Z) \ar[r] & H_{n-q}(M \times S^{2\ell+1}/S^1;\Z) \ar[u]
}
$$
where the vertical maps are the Poincar\'e duality isomorphisms and the horizontal one is the map induced by $f$ on homology.
To define the equivariant push-forward map $f^{!}$ take $\ell$ large enough such that these cohomology spaces are equal to the equivariant cohomology of $M$ and $N$, see Remark \ref{finiteappro}. The push-forward is independent of $\ell$.
This map is sometimes called the {\bf equivariant Gysin homomorphism}. 
 We similarly define the equivariant pushforward map induced by $(S^1)^k$-equivariant maps.

 For an $S^1$-invariant embedded surface $\iota_{\Sigma} \colon \Sigma \to M$ in a four-dimensional $M$, 
the Poincar\'e dual of $\Sigma$ as an equivariant cycle in $M$, i.e.,
$\iota_{\Sigma}^{!}(1)$, where $1 \in H_{S^1}^{0}(\Sigma;\Z)$, is a class in $H_{S^1}^{2}(M;\Z)$. Its pullback under $\iota_{\Sigma}$ equals  the equivariant Euler class $e_{S^1}(\nu(\Sigma \subset M))$ of the normal $S^1$-vector bundle of $\Sigma$ in $M$. 
\end{noTitle}

The Atiyah-Bott/ Berline-Vergne (ABBV) localization formula 
\cite{AB:1984, BV:1982} expresses 
the equivariant pushforward under $\pi \colon M \to \pt$ of an equivariant cohomology class
as a sum of equivariant pushforwards of $\pi|_{F}$ over 
the connected components $F$ of the fixed point set $M^T$ as follows.
We must use $\Q$ coefficients because Euler classes are inverted.

\begin{Theorem}[Atiyah-Bott \cite{AB:1984} / Berline-Vergne \cite{BV:1982}]\label{thm:abbv}
Suppose a torus $T$ acts on a closed manifold $M$.  Then for any class $\alpha\in H_T^*(M;\Q)$,
\begin{equation}\label{eq:abbv}
\pi^{!}(\alpha) = \sum_{{F\subseteq M^T}} \left({\pi|_{F}}\right)^{!}\left( \frac{ \alpha|_F}{e_T(\nu(F\subseteq M))}\right).
\end{equation}
where the sum on the right-hand side is taken over the connected components $F$ of the fixed point set $M^T$, 
$\alpha|_F$ is the restriction of $\alpha$ to $F$, and $e_T(\nu(F\subseteq M))$ is the equivariant Euler class of the normal bundle of $F$.
\end{Theorem}

\section{Formulas for equivariant characteristic classes} \label{se:new}

The {\bf equivariant characteristic classes} of an equivariant vector bundle $E$ are the 
characteristic classes of the vector bundle $\widetilde{E}$ on $(M \times ES^1)/S^1$ whose 
pull-back to $M \times ES^1$ is $E \times ES^1$. These characteristic classes are 
elements of $H_{S^1}^{*}(M)$. In particular, we have the equivariant Euler class $e_{S^1}(E)$ 
when $E$ is an oriented equivariant real vector bundle and the equivariant 
Chern classes $c_{\ell}^{S^1}(E)$ when $E$ is an equivariant complex vector bundle.
As discussed in \cite[\S 5]{Tu-char:2010}, these equivariant characteristic classes are 
{\bf equivariant extensions} of the ordinary characteristic classes.
To interpret the ABBV relation, we first calculate explicitly  equivariant Euler 
classes and their (formal) inverses.
For the Euler classes, we work with integer coefficients.  For their inverses, 
we must revert to $\Q$. 

In 
case $\nu(F \subset M)$ is an equivariant complex vector bundle over a point,
the formula \cite[(C.13)]{Ggk:2002} for the equivariant Euler class of $\nu(\{\pt\} \subset M)$, obtained from  the splitting principle in equivariant cohomology \cite[Theorem C.31]{Ggk:2002},  simplifies to a single term, as follows.

\begin{Lemma}\label{le:euler}
Consider a linear circle action $S^1\acts\, \C^n$ with weights $b_1,\dots,b_n\in \Z \smallsetminus \{0\}$. Thought of as an equivariant bundle over a point $\C^n=\nu(\{ \vec{0}\}\subset \C^n)\to \vec{0}$, 
this has equivariant Euler class
\begin{equation} \label{eq euler0}
e_{S^1}(\C^n) = (-1)^n b_1\cdots b_n u^n\in H_{S^1}^*(\pt;\Z) = \Z[u],
\end{equation}
with (formal) inverse
$$
(e_{S^1}(\C^n))^{-1} = \frac{(-1)^n}{b_1\cdots b_n u^n}\in \Q[u,u^{-1}].
$$
\end{Lemma}

In the case of an equivariant complex line bundle over a positive-dimensional manifold,
where the action fixes the zero-section,
we may also identify the equivariant Euler class explicitly.  
Moreover, in this case, the equivariant
Euler class is invertible (in the appropriate ring), and we have an explicit formula for its inverse.

\begin{Lemma}\label{le:inv euler}
Let
$$
S^1\acts \left(\begin{array}{c} \xymatrix{\mathscr{L}\ar[d]\\ \Sigma}\end{array}\right)
$$
be an equivariant complex line bundle with the zero section fixed pointwise and the fiberwise action linear.
At any point $p\in \Sigma$, let $b\in \Z \smallsetminus \{0\}$ denote the weight of the circle action on the fibre
over $p$.  
Then the equivariant Euler class of $\mathscr{L}$ has the form
\begin{equation}\label{eq euler}
e_{S^1}(\mathscr{L}) = -1 \otimes b\cdot u +  e(\mathscr{L})\otimes 1\in H^2_{S^1}(\Sigma;\Z),
\end{equation}
where $e(\mathscr{L})\in H^2(\Sigma;\Z)$ denotes the ordinary Euler class of $\mathscr{L}$. Its inverse
(in the ring of rational functions with coefficients in $H^*(\Sigma;\Q)$, namely $H^*(\Sigma;\Q)[u,u^{-1}]$) is
\begin{equation}\label{eq inverse euler}
e_{S^1}(\mathscr{L})^{-1} = -\sum_{i=0}^N e(\mathscr{L})^i\otimes  \left(\frac{1}{bu}\right)^{i+1},
\end{equation}
where $N=\Big\lfloor\frac{\dim(\Sigma)}{2}\Big\rfloor$.
\end{Lemma}

\begin{proof}
We first note that because the $S^1$ action fixes the surface $\Sigma$, and since both the cohomology of the orientable surface $\Sigma$ and the cohomology of $BS^1$, over $\Z$, are torsion free,  Kunneth formula gives the splitting 
$$
H_{S^1}^*(\Sigma;\Z) = H^*(\Sigma;\Z)\otimes H^*_{S^1}(\pt;\Z).
$$
Moreover,
$$
e_{S^1}(\mathscr{L})\in H^2_{S^1}(\Sigma;\Z)
$$
and by the splitting,
$$
H^2_{S^1}(\Sigma;\Z)\cong \bigg(H^0(\Sigma;\Z) \otimes H^2_{S^1}(\pt;\Z) \bigg) \bigoplus  \bigg(H^2(\Sigma;\Z)\otimes H^0_{S^1}(\pt;\Z)\bigg).
$$
The leading term in \eqref{eq euler} is guaranteed by  \cite[(C.13)]{Ggk:2002}.  Furthermore, the equivariant Euler class
is defined to be the Euler class of the bundle $\widetilde{\mathscr{L}}$ that fits into the diagram
$$
\xymatrix{
\mathscr{L}\times ES^1 \ar[r]\ar[d] & \widetilde{\mathscr{L}} \ar[d] \\
\Sigma\times ES^1 \ar[r]_(0.4){p} & \Sigma\times_{S^1} ES^1.
}
$$
By naturality of characteristic classes, we must have that the restriction $$p^*(e_{S^1}(\mathscr{L})) = e(\mathscr{L}),$$ and
so the second term in \eqref{eq euler} must be $e(\mathscr{L})\otimes 1$.

To check our formula for $e_{S^1}(\mathscr{L})^{-1}$, we take the product
\begin{eqnarray*}
 \left(-1 \otimes b\cdot u + e(\mathscr{L})\otimes 1\right)& \cdot& \left( -\sum_{i=0}^N e(\mathscr{L})^i \otimes  \left(\frac{1}{bu}\right)^{i+1} \right)\\
& = & 
\sum_{i=0}^N e(\mathscr{L})^i\otimes \left(\frac{1}{bu}\right)^{i} 
 - \sum_{i=0}^N e(\mathscr{L})^{i+1}\otimes \left(\frac{1}{bu}\right)^{i+1}\\
 &  =  & 1\otimes 1 - e(\mathscr{L})^{N+1}\otimes \left(\frac{1}{bu}\right)^{N+1}. 
\end{eqnarray*}
But $e(\mathscr{L})^{N+1}=0$ for dimension reasons, so we see that 
the product equals $1\otimes 1$, as desired.
\end{proof}

 \begin{noTitle} \label{weights}
 {\bf Weights of the action.}
Consider an effective Hamiltonian $S^1\acts (M^4,\omega)$.
At an isolated fixed point $p \in M$, 
there exist complex coordinates on a neighbourhood of $p$ in $M$, and unique non-zero integers $m$ and $n$, such that the circle action is linear with weights $m$ and $n$ 
  \cite[Corollary A.7]{Karshon:1999}; the tangent space $TM|_{p}=\C^2$ splits as a sum of complex line-bundles $\C_{m} \oplus \C_{n}$. Denote the absolute values of the weights at $p$ by $m_p$ and $n_p$.

  At a fixed surface $\Sigma$, we have $TM|_{\Sigma}=T \Sigma \oplus \nu(\Sigma \subset M)$.
  The normal bundle $ \nu(\Sigma \subset M)$ can be viewed as an equivariant complex vector bundle \cite[Corollary A.6]{Karshon:1999}, moreover
  it is a complex {line} bundle since $\dim_{\R}M=4$ and $\dim_{\R}\Sigma=2$. Note that the weight of the action on $T\Sigma$ is $0$.
  For any point $p\in \Sigma$, the $S^1$-weight in the
normal direction to $\Sigma$ is $\pm 1$.  It must be so because the action is effective
and if it were $\pm b$ with $|b|>1$, there would be a global $\Z_b$ stabilizer.
Moreover, it is positive when $\Sigma$ is a minimum and negative when $\Sigma$ is
a maximum.
  \end{noTitle}
For an isolated fixed point $p\in M^{S^1}$, the equivariant Euler class is
an element of  $H^4_{S^1}(p;\Z)$. By Lemma~\ref{le:euler} and \S \ref{weights},
$$e_{S^1}\big(\nu(\{p\}\subset M)\big) =\left\{\begin{array}{rl}
 -m_pn_pu^2 & \text{ if $p$ is interior}\\
 m_p n_p u^2& \text{ if $p$ is extremal}
\end{array}\right.
$$
with inverse 
\begin{equation}\label{eq:eu-11}
\left(e_{S^1}\big(\nu(\{p\}\subset M)\big)\right)^{-1} = \left\{\begin{array}{rl}
 -e_p\frac{1}{u^2} & \text{ if $p$ is interior}\\
 e_p\frac{1}{u^2} & \text{ if $p$ is extremal}
\end{array}\right.
\end{equation}
where  $e_p:=\frac{1}{m_pn_p}$.
For an $S^1$-fixed surface $\Sigma$ (which must be a minimum or maximum
critical set), 
the equivariant Euler class is an element
of $H^2_{S^1}(\Sigma;\Z)$.  
From Lemma~\ref{le:inv euler} and \S \ref{weights}, the equivariant Euler class is
$$e_{S^1}\big(\nu(\Sigma\subset M)\big) = \pm 1\otimes u + e_\Sigma[\Sigma]\otimes 1,$$
where the first sign is determined by whether $\Sigma$ is a minimum ($-$) 
or maximum ($+$), $e_{\Sigma}$
is the self intersection $\Sigma\cdot \Sigma$, and $[\Sigma]$ is the Poincar\'e dual of the class of a point in $H_0(\Sigma;\Z)$. 
(Recall that under the identification $H^2(\Sigma;\Z)\cong \Z$,
the self intersection is the ordinary Euler class of the 
normal bundle $\nu(\Sigma\subset M)$.)  
By Lemma~\ref{le:inv euler}, the inverse is
\begin{equation} \label{eq:eu-12}
e_{S^1}\big(\nu(\Sigma\subset M)\big)^{-1} = \pm1\otimes \frac{1}{u} - e_\Sigma[\Sigma]\otimes \frac{1}{u^2}.
\end{equation}

We may also deduce the restrictions of the equivariant Chern classes to connected components of the fixed point 
set from \S \ref{weights} and the splitting principle. 
\begin{Corollary} \label{cor:chern}
Consider an effective Hamiltonian $S^1 \acts (M^4,\omega)$.
\begin{itemize}
\item At a fixed point $p$ 
$${c_{1}^{S^1}(TM)}|_{p}=(-m-n)u \in H_{S^1}^{2}(p;\Z),$$ ${c_{2}^{S^1}(TM)}|_{p}=m n u^2 \in H_{S^1}^{4}(p;\Z)$, and $$({{c_1^{S^1}}(TM)|_{p})^2-2c_2^{S^1}}(TM)|_{p}=(m^2+n^2) u^2,$$  
where $m,n$ are the weights of the action at $p$.
\item At a fixed surface $\Sigma_{*}$, where $*$ is either $\min$ or $\max$,
$$c_1^{S^1}(TM)|_{\Sigma_{*}}=(2-2g)[\Sigma]\otimes 1 + e_{*}[\Sigma] \otimes 1+(-1)^{\delta_{*=\min}}	 \otimes u,$$
$c_2^{S^1}(TM)|_{\Sigma_{*}}=(-1)^{\delta_{*=\min}}(2-2g)[\Sigma]\otimes u$, 
and 
$$({{c_1^{S^1}}(TM)|_{\Sigma_{*}})^2-2c_2^{S^1}}(TM)|_{\Sigma_{*}}=1 \otimes u^2+2(-1)^{\delta_{*=\min}}e_{*}[\Sigma] \otimes u,$$
where $e_{*}$ is the self intersection of $\Sigma_{*}$.
\end{itemize}
\end{Corollary}

 \begin{proof}
 By \S \ref{weights} and the splitting principle  \cite[Theorem C.31]{Ggk:2002}, the restriction of the 
 total equivariant Chern class
$$
c^{S^1}(TM)=1+c_{1}^{S^1}(TM) +c_{2}^{S^1}(TM) +\ldots
$$
to a connected component $F$ of the fixed point set equals $c^{S^1}(\C_{m}) c^{S^1}(\C_{n})$ if $F=p$ and $c^{S^1}(T \Sigma_*) c^{S^1}(\nu(\Sigma_* \subset M))$ if $F=\Sigma_{*}$.
 The class $c_{1}^{S^1}(L)$ of an equivariant complex line bundle $L$ over a fixed manifold equals $c_{1}(L)-bu$, where $b$ is the weight of the $S^1$-representation on a fiber of $L$, see 
\cite[Example C.41]{Ggk:2002}.
Hence 
$$c^{S^1}(\C_{m}) c^{S^1}(\C_{n})=(1+c_{1}(\C_{m})-m u)(1+c_{1}(\C_{n})-n u)=(1-m u)(1-n u),$$
and 
 $$c^{S^1}(T \Sigma_*) c^{S^1}(\nu(\Sigma_* \subset M))=(1+c_{1}(T\Sigma_*)-0)(1+c_{1}(\nu( \Sigma_* \subset M))-(-1)^{\delta_{*=\max}}\otimes u).$$
The corollary follows; in the case $F=\Sigma_{*}$ we also use the fact that $[\Sigma_*]^2=0$ in the calculations.
 \end{proof}

\section{The equivariant cohomology of a Hamiltonian $S^1$-action on a $4$-manifold: Proof of Theorem~\ref{thm:S1onM4}}\label{se:main thm}

Consider an effective Hamiltonian $S^1$-action on a closed connected four-manifold $M$ with momentum map $\Phi$. 
In the proof of part (A) of Theorem~\ref{thm:S1onM4}, we adapt the argument of Tolman and Weitsman
 \cite[Proof of Prop.~2.1]{TW-hamTsp:1999} 
 to work in our situation when the coefficients are the integers.
 Franz and Puppe have also adapated this argument to work over $\Z$  (see  \cite[Proof of Theorem 5.1]{Franz-Puppe:2007}),
 but they require the action to have connected stabilizer groups, which we do not have.

\subsection*{Proof of Part (A) of Theorem~\ref{thm:S1onM4}}
We will establish that 
$H_{S^1}^*(M;\Z)$ is a free $H_{S^1}^*(\pt;\Z)$-module
isomorphic to $$H^*(M;\Z)\otimes H_{S^1}^*(\pt;\Z).$$

The momentum map $\Phi$ is Morse-Bott at every connected component of the
critical set.  The critical sets are
precisely the connected components of the fixed point set. 
 The fixed point 
set $M^{S^1}$ consists of isolated points and up to two surfaces.  The surfaces
are symplectic submanifolds and are hence orientable.
Thus, $H^*(M^{S^1};\Z)$ is torsion free. 
The negative normal bundle to a critical set $C$ is a complex bundle, except at
the minimum, where the negative normal bundle is a rank zero bundle.   Thus, the
Morse-Bott indices are always even and hence $\Phi$ is a perfect Morse-Bott function.

Let $c$ be a critical value for $\Phi$ and denote
$$
M^\pm = \Phi^{-1}((-\infty,c\pm\varepsilon)).
$$
Consider
the long exact sequence in equivariant cohomology with $\Z$
coefficients,
$$
\cdots\to H^{*}_{S^1}(M^+,M^-)\to H^*_{S^1}(M^+) \to H^*_{S^1}(M^-) \to
H^{*+1}_{S^1}(M^+,M^-)\to\cdots.
$$
First suppose that $c=\Phi(C)$ is a non-minimal critical value corresponding to a critical set $C$.  We may 
choose $\varepsilon$ to be small enough that $c$ is the only critical value in the interval
$[c-\varepsilon,c+\varepsilon]$.  Denote the negative disc and sphere
bundles to $C$ in $M$ by $D_c$ and $S_c$ respectively.
We let $\lambda$ denote the Morse index of $C$ in $M$.  Following an
identical argument to \cite[Proof of Proposition 2.1]{TW-hamTsp:1999},  we
obtain a commutative diagram, with $\Z$ coefficients,
\begin{equation}\label{eq:commdiagr}
\begin{array}{c}
\xymatrix{
{}\ar[r] &  H^{*}_{S^1}(M^+,M^-)\ar[r]^{\gamma_c}\ar[d]_{\cong} &
H^*_{S^1}(M^+) \ar[r]^{\beta_c}\ar[d] & H^*_{S^1}(M^-)\ar[r] &  \\
&  H_{S^1}^*(D_c,S_c)\ar[r]^{\delta_c}\ar[d]_{\cong} & H_{S^1}^*(D_c) & & \\
&  H^{*-\lambda}_{S^1}(D_c) \ar[ur]_{\cup e_c} & & &
}\end{array}
\end{equation}
where $e_c=e_{S^1}(D_c)$ is the equivariant Euler class of
the bundle $D_c\to C$. The left-most vertical arrows in the diagram are
excision and the Thom isomorphism with $\Z$
coefficients.  An explicit analysis of the Thom isomorphism and the push-pull forumla
guarantee that the diagonal arrow is
indeed the cup product with $e_c$. 
Because  $C\subseteq M$ is torsion-free\footnote{
This is the point at which Franz and Puppe 
\cite[Proof of Theorem 5.1]{Franz-Puppe:2007}
need connected stabilizer groups to guarantee
that the Euler class $e_c$ is primitive and doesn't interfere with any torsion in the fixed components.
By contrast, we have no torsion.
}, the cup
product map $\cup e_c$ is injective, and so $\delta_c$ and
$\gamma_c$ must also be injective.  Thus, the long exact sequence
splits.

If $a=f(C)$ is the minimum critical value, the spaces $M^+$ and $C$
are equivariantly homotopic, and $M^-$ is empty.  Thus,
$H_{S^1}^*(M^-)=0$ and $H_{S^1}^*(M^+)\cong H_{S^1}^*(C)$.  Therefore,
the sequence splits in this case as well.    Finally, when $b=f(C)$ is the maximum
critical value, $M^+=M$ and we will have assembled $H^*_{S^1}(M;\Z)$.

Because $\Phi$ is a perfect Morse-Bott function for both ordinary
and equivariant cohomology, we may
assemble $H_{S^1}^*(M;\Z)$ as a module over $H_{S^1}^*(\pt;\Z)$.
As the critical set is precisely the set of fixed points, 
we conclude that  if $F_1,\dots, F_s$ are the (finitely many) components of the fixed set $M^{S^1}$
and $\lambda_1,\dots,\lambda_s$ are their Morse-Bott indices, we have now established that as modules
over $H_{S^1}^*(\pt;\Z)$,
$$
H_{S^1}^*(M;\Z) = \bigoplus_j H^{*-\lambda_j}(F_j;\Z)\otimes H_{S^1}^*(\pt;\Z) = H^*(M;\Z)\otimes H_{S^1}^*(\pt;\Z).
$$

\subsection*{Proof of Part (B) of Theorem~\ref{thm:S1onM4}}
 As before, we use $\Phi$ as a Morse-Bott function on $M$. The critical set is
the fixed point 
set $F=M^{S^1}$ and it consists of isolated points and up to two surfaces. 
Order the critical values of
$\Phi$ as $c_1<\cdots<c_N$, and let $F_i$ be the critical points
with $\Phi(F_i)=c_i$. Denote
$$
M_i^\pm = \Phi^{-1}((-\infty,c_i\pm\varepsilon)).
$$
The injectivity statement holds for $M_1^-$ because this set is empty.  We use this as a base
case and proceed by induction.  Suppose that
$$
H_{S^1}^*(M_i^-)\hookrightarrow H_{S^1}^*(M_i^-\cap F)
$$
is an injection with $\Z$ coefficients.  
As in the proof of part (A), we have a short exact sequence  with $\Z$ coefficients
$$
0\to H^{*}_{S^1}(M_i^+,M_i^-)\longrightarrow H^*_{S^1}(M_i^+) \longrightarrow H^*_{S^1}(M_i^-) \to 0.
$$
Let $i^*_\pm$ be the inclusion $M_i^\pm\cap F\hookrightarrow M_i^\pm$.  Then we
have a commutative diagram, with $\Z$ coefficients,
$$
\xymatrix{
0\ar[r] & H^{*}_{S^1}(M_i^+,M_i^-)\ar[r]\ar[d]^{\cong} & H^*_{S^1}(M_i^+)
\ar[r]\ar[d]^{i^*_+} & H^*_{S^1}(M_i^-) \ar[r]\ar[d]^{i^*_-} & 0 \\
0\ar[r] & H^{*}_{S^1}(F_i)\ar[r] & H^*_{S^1}(M_i^+\cap F) \ar[r] &
H^*_{S^1}(M_i^-\cap F) \ar[r] & 0.
}
$$
The map $i_-^*$ is an injection by the inductive hypothesis.  
The Four Lemma guarantees that $i_+^*$ is an injection.  Finally, notice
that $M_i^+$ is equivariantly homotopy equivalent to $M_{i+1}^-$.
This completes the proof of part (B).

\subsection*{Proof of Part (D) of Theorem~\ref{thm:S1onM4}}
Let $M_{(i)}$, $-1 \leq i \leq k$, be the equivariant {\bf $i$-skeleton} of $T=T^k \acts M$, i.e., the union of orbits of dimension $\leq i$. In particular, $M_{(-1)}=\emptyset, \, M_{(0)}=M^{T}$ and $M_{(k)}=M$.

By \cite[Lemma 4.1]{Franz-Puppe:2007}, it is enough to show that 
for  $0 \leq j \leq k=1$ the long exact
 sequence, with $\Z$ coefficients,
 $$
  \cdots \to H_{S^1}^{*}(M,M_{(j)}) \to H_{S^1}^{*}(M,M_{(j-1)}) \to H_{S^1}^{*}(M_{(j)},M_{(j-1)}) \to H_{S^1}^{*+1}(M,M_{(j)}) \to \cdots,$$
  obtained from the inclusion of pairs $(M_{(j)},M_{(j-1)}) \hookrightarrow (M,M_{(j-1)})$,
  splits into short exact sequences (over $\Z$)
  $$0 \to H_{S^1}^{*}(M,M_{(j-1)}) \to H_{S^1}^{*}(M_{(j)},M_{(j-1)}) \to H_{S^1}^{*+1}(M,M_{(j)}) \to 0,$$  
  i.e., for $j=0$ we need (over $\Z$)
\begin{equation*} \label{eq0}
0 \to H_{S^1}^{*}(M,\emptyset) \to H_{S^1}^{*}(M^{S^1},\emptyset) \to H_{S^1}^{*+1}(M,M^{S^1}) \to 0
\end{equation*}
 and  for $j=1$ we need (over $\Z$)
\begin{equation*} \label{eq1}
0 \to H_{S^1}^{*}(M,M^{S^1}) \to H_{S^1}^{*}(M,M^{S^1}) \to H_{S^1}^{*+1}(M,M) \to 0.
\end{equation*}
For $j=0$, we have the long exact sequence for the pair $(M,M^{S^1})$.
The injectivity of the map $i^*:H_{S^1}^*(M)\to H_{S^1}^*(M^{S^1})$, proven above,
then forces the long exact sequence to split into short exact sequences, as desired.
In the case $j=1$ the long exact sequence degenerates, so it splits into very short exact sequences: $H_{S^1}^{*+1}(M,M)=0$ 
and the map $ H_{S^1}^{*}(M,M^{S^1}) \to  H_{S^1}^{*}(M,M^{S^1})$ induced from the 
inclusion $(M,M^{S^1}) \to (M,M^{S^1})$ is the identity isomorphism.

\subsection*{Proof of Part (E) of Theorem~\ref{thm:S1onM4}}
Denote by  $I \colon M \to (M\times ET)/T$ the inclusion of the fiber, and by 
\begin{equation} \label{surjfiber}
 I^{*} \colon H_{T}^{*}(M;\Z) \to H^{*}(M;\Z)
\end{equation} 
the induced map on cohomology.
Franz and Puppe have established that the exactness of the Atiyah-Bredon sequence implies that 
$I^{*} \colon H_{S^1}^{*}(M;\Z) \to H^{*}(M;\Z)$ is a surjection  \cite[Theorem 1.1]{Franz-Puppe:2007}.
This is equivalent to the statement that the map 
$$H_{S^1}^{*}(M;\Z)\otimes_{H_{S^1}^*} \Z \to H^{*}(M;\Z)  $$
is an isomorphism.  Moreover, because $H_{S^1}^*(\pt;\Z)$ is generated by the degree two class $u$, we then 
have a short exact sequence
$$0 \to \langle \pi^{*}(u) \rangle \to  H_{S^1}^{*}(M;\Z) \to H^{*}(M;\Z) \to 0.$$

\begin{noTitle}\label{ranks}
{\bf Equivariant Poincar\'e polynomials.} 
The $S^1$-equivariant cohomology  of $M$ splits 
$$
H_{S^1}^*(M;R) \cong  H^*(M;R)\otimes H^*_{S^1}(\pt;R),
$$
as $H^*_{S^1}(\pt;R)$-modules: 
Theorem \ref{thm:module} for $R=\Q$ and Part (A) of Theorem~\ref{thm:S1onM4} for $R=\Z$.
Hence the equivariant Poincar\'e polynomial splits:
\begin{equation}
P_{S^1}^{M}(t)=P^{M}(t)\cdot P_{S^1}^{\pt}(t).
\end{equation}
By \eqref{eqhs1pt},
\begin{equation}
P_{S^1}^{\pt}(t)=(1+t^2+t^4+\ldots)=\frac{1}{1-t^2}.
\end{equation}
We use Morse theory to find the Poincar\'e polynomial $P^{M}(t)$.
The momentum map of the Hamiltonian circle action is a perfect Morse-Bott function whose critical points 
are the fixed points for the circle action \cite[\S 32]{GS-book:1984}.
Therefore
 \begin{equation}\label{eqHj}
 \dim(H^{j}(M;R))=\sum \dim H^{j-\lambda_F}(F;R),
 \end{equation}
 where we sum over the connected components $F$ of the fixed point set, and
where $\lambda_F$ is the index of the component $F$.
In the special case of $S^1\acts M^4$, the fixed point components are
finitely many isolated points and up to two orientable surfaces, of the same genus $g$.
The contribution of each fixed  component is as follows. For a minimal fixed surface, $\lambda_F=0$, 
so the contribution to $H^{0}(M)$ is $1$ and to $H^{1}(M)$ is $2g$.
For a maximal fixed surface $\lambda_F=2$, so the contribution to $H^{2}(M)$ is $1$ and to $H^{3}(M)$ is $2g$.
For an isolated fixed point, $\lambda_F=0$ if it is minimal, and the point contributes $1$ to $H^{0}(M)$.
For an isolated fixed point, $\lambda_F=2$ if it is an interior fixed point, and the point contributes $1$ to $H^{2}(M)$.
Finally, for an isolated fixed point, $\lambda_F=4$ if it is maximal,  and the point contributes $1$ to $H^{4}(M)$.
See, e.g., \cite{YK-max_tori:2003}. 
\end{noTitle}

Hence
\begin{equation}\label{eqpoly1}
P^M(t)=1+\delta_{min} 2 g t+ (\ell-2+2\delta_{min}+2\delta_{max}) t^2+
\delta_{max} 2g t^3+t^4,
\end{equation}
and 
\begin{eqnarray}\label{eqfixedp}
P^{M^{S^1}}(t) &=& \big|\text{ isolated points in }{M^{S^1}}\ \big|+\big|\text{ surfaces in }{M^{S^1}}\ \big|(1+2gt+t^2) \\ \nonumber
                     &=&(\ell+\delta_{min}+\delta_{max})+(\delta_{min}+\delta_{max})2gt+(\delta_{min} +\delta_{max})t^2,
\end{eqnarray}
where
$$\ell=\# \text{ isolated fixed points},$$
$$\delta_{min}=\# \text{ minimal fixed surfaces (zero or one)}, \mbox{ and}$$
$$\delta_{max} =\# \text{ maximal fixed surfaces (zero or one)}.$$
Therefore
\begin{equation}\label{eq:main}
\begin{array}{rcl}
P_{S^1}^{M}(t) & = & P^{M}(t) \cdot  \frac{1}{1-t^2}\\
& = & 1 + (\ell-1+2\delta_{min}+2\delta_{max})t^2 + 
(\ell+2\delta_{min}+2\delta_{max})t^4\left(\frac{1}{1-t^2}\right)\\
& &  + \delta_{min}2g t +(\delta_{min}+\delta_{max})2gt^3\left(\frac{1}{1-t^2} \right).
\end{array}
\end{equation}
Also
\begin{equation}\label{eq:main2}
\begin{array}{rcl}
P_{S^1}^{M^{S^1}}(t)&=&P^{M^{S^1}}(t) \cdot \frac{1}{1-t^2}\\
&=&(\ell+\delta_{min}+\delta_{max})+(\ell+2\delta_{min}+2\delta_{max})t^2\\
& & +(\ell+2\delta_{min}+2\delta_{max})t^4\left(\frac{1}{1-t^2}\right)\\
& &+(\delta_{min}+\delta_{max})2gt+(\delta_{min}+\delta_{max})2gt^3
\left(\frac{1}{1-t^2} \right).
\end{array}
\end{equation}

The differences between the corresponding coefficients in $P_{S^1}^{M^{S^1}}(t)$ and $P_{S^1}^{M}(t)$ tell us how many constraints  
cut out $i^*(H_{S^1}^*(M;R))$ in $H^*_{S^1}(M^{S^1};R)$.  
The constraints are linear relations
among the equivariant cohomology classes on $M^{S^1}$, and we will refer colloquially to
the {\bf relations} the classes must satisfy.
Equations \eqref{eq:main} and \eqref{eq:main2} combine to give 
the following lemma.

\begin{Lemma}\label{le:relations}
Let $S^1\acts M^4$ be a closed connected Hamiltonian space.
\begin{eqnarray*}
P_{S^1}^{M^{S^1}}(t)-P_{S_1}^{M}(t)&=&\Big[ (\ell+\delta_{max}+\delta_{min}-1)+\delta_{max} 2 gt+(2-\ell-\delta_{max}-\delta_{min})t^2\\
& & \phantom{MOVE OVER a bit} -\delta_{max}2g t^3-t^4\Big]\cdot (1+t^2+t^4+\ldots)\\
&=&(\# \text{ components of }M^{S^1}-1)+2gt+t^2.
\end{eqnarray*}
The coefficient $2g$ of $t$ in the last equality follows because 
if $g>0$, we must have $\delta_{max}=1$.
Hence we must find $(\# \text{ components of }M^{S^1}-1)$ relations in degree $0$;  $2g$ relations in degree $1$; and one relation in degree $2$ to 
determine the image $$i^*(H_{S^1}^*(M;R))\subset 
H^*_{S^1}(M^{S^1};R),$$ for $R=\Z,\, \Q$. 
\end{Lemma}

\begin{noTitle}\label{sec:ABBV}
{\bf Constraints on the image of $i^*$ from ABBV.}
The ABBV relation \eqref{eq:relations} will impose one constraint in homogeneous degree $2$ on tuples
$$
\alpha=(\alpha|_F)\in H^*_{S^1}(M^{S^1};\Q) = \bigoplus_{F\subset M^{S^1}} 
H_{S^1}^*(F;\Q).
$$
Suppose we have such a tuple $\alpha=(\alpha|_F)$.  For an isolated fixed point $p$, set
$$
\delta_p=\begin{cases}
1 & p\text{ is interior}\\
0 & p\text{ is extremal}.
\end{cases}$$
At each isolated fixed point $p$, we may
identify
$H^2_{S^1}(p;\Q)=H^0(p;\Q)\otimes H^2_{S^1}(\pt;\Q)$.  Thus,
$\alpha|_p=1\otimes c_pu$ for some $c_p\in \Q$.  
In the ABBV relation \eqref{eq:relations}, this will contribute 
$$
(\pi|_p)^!\Big[\alpha|_p\cdot (e_{S^1}(\nu(p\subset M)))^{-1}\Big]=(\pi|_p)^!\left[(1\otimes c_pu) \cdot \left((-1)^{\delta_{p}}e_p\otimes \frac{1}{u^2}\right)\right]= (-1)^{\delta_{p}}\frac{c_pe_p}{u},
$$
where the first equality is by \eqref{eq:eu-11}.
Next, for a fixed surface $\Sigma$, 
$$
H^2_{S^1}(\Sigma;\Q)=(H^2(\Sigma;\Q)\otimes H^0_{S^1}(\pt;\Q))\oplus (H^0(\Sigma;\Q)\otimes H^2_{S^1}(\pt;\Q)).
$$
Thus $\alpha|_{\Sigma}=[\Sigma]\otimes a_{\Sigma} + b_{\Sigma}\otimes u$, where
$a_{\Sigma},b_{\Sigma}\in\Q$.
In the term in \eqref{eq:relations}, this will contribute 
\begin{eqnarray*}
(\pi|_\Sigma)^!\Big[\alpha|_\Sigma \cdot (e_{S^1}(\nu(\Sigma\subset M)))^{-1}\Big] & = &(\pi|_\Sigma)^!\Big[([\Sigma]\otimes a_{\Sigma} + b_{\Sigma}\otimes u)\cdot 
(e_{S^1}(\nu(\Sigma\subset M)))^{-1}\Big] \\ 
& = &
(\pi|_\Sigma)^!\left[([\Sigma]\otimes a_{\Sigma} + b_{\Sigma}\otimes u) \cdot \left( \pm 1\otimes \frac{1}{u} - e_\Sigma[\Sigma]\otimes \frac{1}{u^2} \right)\right]\\
& = & \pm \frac{a_\Sigma}{u}-\frac{b_\Sigma e_\Sigma}{u},
\end{eqnarray*}
where the second equality is by \eqref{eq:eu-12}.
Combining these, we get a term of the form
$$
\sum_p  (-1)^{\delta_{p}}\frac{c_pe_p}{u} +\delta_{max} \left(\frac{a_{max}}{u}-\frac{b_{max} e_{max}}{u}\right) -\delta_{min} \left(\frac{a_{min}}{u}+\frac{b_{min} e_{min}}{u}\right),
$$
where $p$ runs over all isolated fixed points.
This is in $\Q[u]$ if and only if
\begin{equation}\label{eq:abbv relation}
\left(\sum_p  (-1)^{\delta_{p}}c_pe_p\right) +\delta_{max} \left(a_{max}-b_{max} e_{max}\right) -\delta_{min} \left(a_{min}+b_{min} e_{min}\right)=0.
\end{equation}
This precisely gives us one linear relation among the rational numbers $c_p$, $a_{max}$,
$b_{max}$, $a_{min}$ and $b_{min}$.
\end{noTitle}

\subsection*{Proof of  Part (C) of Theorem~\ref{thm:S1onM4}}
We want to determine which classes in $H^*_{S^1}(M^{S^1};\Q)$ are the images
of global equivariant classes.  Let $S$ denote the submodule of classes in
$H^*_{S^1}(M^{S^1};\Q)$ which satisfy conditions (0), (1)  and (2) of
Theorem~\ref{thm:S1onM4}.
As a submodule of a free module over the PID $\Q[u]=H_{S^1}^*(\pt;\Q)$, 
the submodule $S$ is itself a free  $H_{S^1}^*(\pt;\Q)$-module.

By Parts (A) and (B) of Theorem~\ref{thm:S1onM4} and Theorem~\ref{thm:module}, we
also know that $i^*(H_{S^1}^*(M;\Q))$ is a free submodule of $H^*_{S^1}(M^{S^1};\Q)$.
We aim to show that $i^*(H_{S^1}^*(M;\Q))\subset S$ and 
that these have equal rank in homogeneous degree $k$ for each $k$.
This will prove that $i^*(H_{S^1}^*(M;\Q))= S$.

We first consider equivariant cohomology classes of homogeneous degree zero.  
By Theorem~\ref{thm:module}, we have 
$H_{S^1}^0(M;\Q) = H^0(M)\otimes H_{S^1}^0(\pt;\Q) = \Q\otimes H_{S^1}^{0}(\pt;\Q)$,
and so
$$
\dim(H_{S^1}^0(M;\Q))=1
$$
over $H_{S^1}^0(\pt;\Q)$.
Constant functions on $M$ are equivariant, so they represent classes in
$H_{S^1}^0(M;\Q)$.  They must represent all of $H_{S^1}^0(M;\Q)$ since
it is one dimensional.  Thus, for $\alpha\in i^*(H_{S^1}^0(M;\Q))$, its restriction
to any fixed component is its constant value.  This means that for a class
in $H_{S^1}^0(M^{S^1};\Q)$ to be in the image of $i^*$, it must be a constant tuple,
which is equivalent to satisfying $(\# \text{ components of }M^{S^1}-1)$ relations
which force the tuple to be constant.  These are the $(\# \text{ components of }M^{S^1}-1)$
relations sought in Lemma~\ref{le:relations}.

Next, we note that because the action $S^1\acts M^{S^1}$ is trivial,
$$H^1_{S^1}(M^{S^1};\Q) = \Big(H^1(M^{S^1};\Q)\otimes H^0_{S^1}(\pt;\Q)\Big)
\oplus \Big(H^0(M^{S^1};\Q)\otimes H^1_{S^1}(\pt;\Q)\Big),$$
and the second term on the right-hand side is zero since $BS^1$ is simply connected.  
Moreover, $H^1(M^{S^1})$ is non-zero 
only if there are fixed surfaces of positive genus.  Thus, $H^1_{S^1}(M^{S^1})$
is non-zero only in the positive genus case.  In that case, we have two fixed surfaces of
the same genus.  A homogeneous 
equivariant class of degree
$1$ will be zero on each interior fixed point.  A globally constant class of homogenous
degree one is, as ever, 
$S^1$-equivariant.  Such a class will restrict to the same class on $\Sigma_{max}$
and $\Sigma_{min}$.  That is, we will have a pair $(\alpha|_{\Sigma_{min}},
\alpha|_{\Sigma_{max}})$ for which, when we identify 
$H^1(\Sigma_{min};\Q)\cong H^1(\Sigma_{max};\Q)$, we have 
$\alpha|_{\Sigma_{min}}=\alpha|_{\Sigma_{max}}$.
The possible classes of this form make up a $\dim(H^1(\Sigma;\Q))=2g$ dimensional
subspace of $H^1_{S^1}(M^{S^1})$. We know from $\eqref{eq:main}$ that
$i^*(H^1_{S^1}(M;\Q))$ is $2g$ dimensional, so as in the degree zero case, these must 
be everything in the image of $i^*$.  In terms of relations, we will have exactly
the $4g-2g=2g$ relations that force $\alpha|_{\Sigma_{min}}=\alpha|_{\Sigma_{max}}$,
namely the $2g$ relations sought in Lemma~\ref{le:relations}. 

We conclude that $i^*(H_{S^1}^*(M;\Q))$ is a subset of the submodule of classes in
$H^*_{S^1}(M^{S^1};\Q)$ which satisfy conditions (0) and (1). 
By the ABBV localization formula \ref{thm:abbv}, every class in $i^*(H_{S^1}^*(M;\Q))$ is in the submodule of classes that satisfy the ABBV relation \eqref{eq:relations}. Therefore, $i^*(H_{S^1}^*(M;\Q))$ is a subset of the intersection submodule $S$.
The ABBV relation imposes weaker constraints than being globally constant in
homogenous degree zero, and it imposes no constraints in homogenous degree
one.  
In homogeneous degree two, it imposes exactly one constraint \eqref{eq:abbv relation}. 
This is precisely the one degree two relation sought in Lemma~\ref{le:relations}.

\noindent Thus, we have verified that $i^*(H_{S^1}^*(M;\Q))\subset S$ and has the same
ranks, so the two (free) submodules must be equal.  This completes the proof of 
Theorem~\ref{thm:S1onM4}. \hfill \qed

To assemble the equivariant cohomology of a complexity one space, we will need
a slightly more general form of Theorem~\ref{thm:S1onM4}.    
We consider 
a Hamiltonian $T$-action on a symplectic four-manifold $M$ which is the extension
of a Hamiltonian $S^1\acts M$ by a trivial action 
of a subtorus $K$ of codimension one.  This forces the fixed point set
$M^T$ to consist of isolated points and up to two surfaces.  We still have 
the parameters associated to the decorated graph
described in Section~\ref{se:dec graphs} for  the Hamiltonian $T/K$-action.

\begin{Proposition}\label{prop:T on M4}
Let $M$ be a closed connected symplectic four-manifold. Let a torus $T$ act 
non-trivially in a Hamiltonian fashion on $M$, 
and suppose that a codimension one subtorus $K\subset T$ acts
trivially.  Let $\pi_K: H_T^*(-;R)\to H_K^*(-;R)$ be the restriction
map in equivariant cohomology.
\begin{enumerate}
\item[(A)]
The equivariant cohomology,
$H_{T}^*(M;\Z)$ is a free $H_{T}^*(\pt;\Z)$-module
isomorphic to $H^*(M;\Z)\otimes H_{T}^*(\pt;\Z)$.

\item[(B)]
The inclusion $i:M^{T}\hookrightarrow M$ induces an injection in 
integral equivariant cohomology
$$
i^*:H^*_{T}(M;\Z) \hookrightarrow H^*_{T}\left(M^{T};\Z\right).
$$

\item[(C)]
In equivariant cohomology with rational coefficients, the image of $i^*$ is characterized
as those classes  $\alpha\in H^*_{T}\left(M^{T};\Q\right)$ which satisfy:
\begin{enumerate}
\item[(1)] $\pi_K(\alpha|_{F_i}) = \pi_K(\alpha|_{F_j})$ for all components 
$F_i,F_j$ of $M^T$; and

\item[(2)] the ABBV relation
\begin{equation}
 \sum_{{F\subset M^{T}}}\left({\pi|_F}\right)^{!}\left( \frac{ \alpha|_F}{e_{T/K}(\nu(F\subseteq M))}\right) \in H_T^*(\pt;\Q),
\end{equation}
where the sum is taken over connected components $F$ of the fixed point set $M^{T}$, 
and the equivariant Euler 
classes are taken with respect to the effective $T/K$-action.
\end{enumerate}

\item[(D)]  The Atiyah-Bredon sequence
for $T\acts M$
is exact over $\Z$.

\item[(E)]
The map \eqref{surjfiber}
is a surjection $H_{T}^{*}(M;\Z) \to H^{*}(M;\Z)$.  
\end{enumerate}
\end{Proposition}

\noindent The proof is identical to the proof of Theorem~\ref{thm:S1onM4}.
Parts (A) and (B) follow word-for-word as above, using $\Phi$ as a Morse-Bott function.
For part (C), the conditions (0) and (1) in  Theorem~\ref{thm:S1onM4} have become the single 
more compact form (1) in this generalization.  Writing the conditions in this way
is equivalent to saying that $\alpha|_{F_i}-\alpha|_{F_j}$ is in $\ker(\pi_K)$.
This boils down to requiring that $\alpha|_{F_i}-\alpha|_{F_j}$ is a multiple
of $1\otimes \tau$, where $\tau$ is a generator for the dual of the Lie algebra
of $T/K$.  Since the generator $\tau$ is an element of $H_{T/K}^2(\pt;\Q)$, 
this imposes the requirement
in degrees zero and one that the classes $\alpha|_{F_i}$ and $\alpha|_{F_j}$
agree.
For part (D), the fact that a codimension one subtorus $K\subset T$ acts
trivially guarantees that there are only $0$- and $1$-dimensional orbits, so the
result again follows in a straight forward way from  \cite[Lemma 4.1]{Franz-Puppe:2007}. 
Part (E) is then an immediate consequence of (D).

\begin{Remark}
When the fixed point set consists of isolated points, items (B) and (C) in Theorem~\ref{thm:S1onM4}
coincide with \cite[Proposition~3.1]{Goldin-Holm:2001}, and items (B) and (C) in Proposition~\ref{prop:T on M4}
coincide with \cite[Proposition~3.2]{Goldin-Holm:2001}.  In the case when there are
fixed surfaces, even with genus $0$, our results are already providing new calculations.  
\end{Remark}

\section{The equivariant cohomology of complexity one spaces} \label{se:complexity one}

Recall that a { complexity one space} is a symplectic manifold equipped 
with an effective Hamiltonian action of a torus which is one dimension
less than half the dimension of the manifold.  That is, the torus is one 
dimension too small for the action to be toric.
These manifolds have been classified in terms of combinatorial 
and homotopic data, together called a {\bf painting},
by Karshon and Tolman \cite{KT:2014}.
We begin with a Lemma that establishes
that the stabilizer groups for complexity one spaces must be of the form, a torus cross a cyclic group.

\begin{Lemma}\label{lem:local normal form}
Let $T^{n-1}\acts M^{2n}$ be a complexity one space and $p\in M$.  
Then the stabilizer group $H$ of $p$ is a direct product $H_0\times \Z/k\Z$, where
$H_0$ is a connected torus (the connected component of the identity in $H$).
\end{Lemma}

\begin{proof}
Let $p\in M$ and let $H$ denote the stabilizer of $p$.  The local normal form theorem \cite{GS-normal:1984,Marle:1985} 
for momentum maps guarantees that there is a neighborhood of the orbit $T\cdot p$ that is isomorphic to an open neighborhood
of $[\mathds{1}:0:0]$ in $T\times_H \C^{\ell} \times \mathfrak{ann(h)}$ where $h$ is the (real) dimension of $H$. Counting
dimensions, we must have
$$
2n = (n-1) + 2\ell - h +(n-h-1).
$$
This forces $\ell = h+1$ and means that $H\hookrightarrow (S^1)^\ell= (S^1)^{h+1}$.
This means that there is a short exact sequence
$$
1\to H \to (S^1)^{h+1} \to S^1 \to 1.
$$
Now let $H_0$ denote the connected component of $H$ containing the identity.  
We can quotient the first two terms in the short exact sequence by $H_0$ to get
$$
1\to H/H_0 \to (S^1)^{h+1}/H_0 \to S^1 \to 1.
$$
The fact that $H/H_0$ is finite means that $(S^1)^{h+1}/H_0$ is a finite cover of a circle, so it is a circle itself.  Thus,
$H/H_0$ is a finite subgroup of a circle, so $H/H_0$ is cyclic.

We now want to show that $H= H_0\times H/H_0$.  Suppose that $H/H_0\cong \Z/k\Z$ and choose an element $a\in H$ that
generates $H/H_0$.  Then $a^k\in H_0$.

Now, $H_0$ is a connected abelian group so it must be isomorphic to a torus $(S^1)^h$.  In a torus, every element has a
$k^{\mathrm{th}}$ root.  Choose a $k^{\mathrm{th}}$ root $b\in H_0$ of $a^k$.  That is, $a^k = b^k$.  But now, $ab^{-1}$
generates a cyclic subgroup of $H$ of order $k$ complementary to $H_0$.  That is,
$H \cong H_0\times \langle ab^{-1}\rangle$, which is what we wanted to show.
\end{proof}

We now consider a complexity one space $T^{n-1}\acts M^{2n}$.
For any (closed, connected) subtorus $T^k\subset T^{n-1}$, the set of fixed points
$M^{T^k}$ is a symplectic submanifold with components of dimension 
at most $2n-2k$.  
Depending on how generic $T^k$ is inside of $T^{n-1}$, some components of $M^{T^k}$ could be 
fixed by all of $T^{n-1}$.
If there is a component $X$ for which $(T^{n-1}/T^k)\acts X$ is an effective action, that component $X$ 
must have dimension at least $2n-2k-2$.  
That is, $(T^{n-1}/T^k)\acts X$ will be either toric or complexity one itself.
In particular, the components of the
one-skeleton are two-dimensional or four-dimensional.

Tolman and Weitsman considered the inclusion of the fixed points 
$j: M^{T} \to M_{(1)}$, which is an equivariant map. The following theorem is a consequence
of work by Chang and Skjelbred \cite{CS:1974}.  Tolman and Weitsman give an elementary 
geometric proof.
\begin{Theorem}[Tolman-Weitsman \cite{TW-hamTsp:1999}]\label{thm:TW}
Let $M$ be a closed Hamiltonian $T$-space.
The induced maps in equivariant cohomology with rational coefficients,
$$
\xymatrix{
H_T^*(M;\Q)_{\phantom{BB}} \ar@{^(->}[rd]_(0.4){i^*} & & H_T^*(M_{(1)};\Q)\ar[ld]^(0.4){j^*}\\
 & H_T^*(M^T;\Q) & \\ 
}
$$
have the same image in $H_{T}^*(M^T;\Q)$.
\end{Theorem}
In other words, for a tuple of equivariant
classes on the fixed point components to be in 
the image of $i^*$, we only need to ensure that the tuple 
is a global class on each component of the one-skeleton.
Combining Theorem~\ref{thm:TW}
with our result in Proposition~\ref{prop:T on M4} 
gives the following combinatorial description of the equivariant cohomology
of a complexity one space.

\begin{Corollary}\label{thm:complexity one}
Let $T^{n-1}\acts M^{2n}$ be a closed connected complexity one space.
\begin{enumerate}
\item[(A)]
The equivariant cohomology
$H_{T}^*(M;\Z)$ is a free $H_{T}^*(\pt;\Z)$-module
isomorphic to $H^*(M;\Z)\otimes H_{T}^*(\pt;\Z)$.

\item[(B)]
The inclusion $i:M^{T}\hookrightarrow M$ induces an injection in 
integral equivariant cohomology
$$
i^*:H^*_{T}(M;\Z) \hookrightarrow H^*_{T}\left(M^{T};\Z\right).
$$

\item[(C)]
In equivariant cohomology with rational coefficients, the image of $i^*$ is characterized
as those classes  $\alpha\in H^*_{T}\left(M^{T};\Q\right)$ which satisfy,
for every codimension one subtorus $K\subset T$ and every connected component
$X$ of $M^K$,
\begin{enumerate}
\item[(1)] $\pi_K(\alpha|_{F_i}) = \pi_K(\alpha|_{F_j})$ for all components 
$F_i,F_j$ of $X^T$; and

\item[(2)] when $\dim(X)=4$, the ABBV relation,
\begin{equation}
 \sum_{{F\subset X^{T}}} \left({\pi|_{F}}\right)^{!}\left( \frac{ \alpha|_F}{e_{T/K}(\nu(F\subseteq X))}\right) \in H_T^*(\pt;\Q),
\end{equation}
where the sum is taken over connected components $F$ of the fixed point set $X^{T}$, 
and the equivariant Euler 
classes are taken with respect to the effective $T/K$-action.
\end{enumerate}

\item[(D)]  The Atiyah-Bredon sequence
for $T\acts M$
is exact over $\Z$. 

\item[(E)]
The map \eqref{surjfiber}
is a surjection $H_{T}^{*}(M;\Z) \to H^{*}(M;\Z)$.  
\end{enumerate}

\end{Corollary}

\begin{proof}
Parts (A) and (B) hold for precisely the same reasons as (A) and (B) in Theorem~\ref{thm:S1onM4}.  The components of the fixed set
$M^{T}$ are still zero- and two-dimensional symplectic submanifolds of $M$ and are hence torsion free.
A generic component of the moment map $\Phi^\xi$ will still be a perfect Morse-Bott function with critical set $M^T$.
This will allow us to deduce that $H_{T}^*(M;\Z)$ is a free $H_{T}^*(\pt;\Z)$-module
isomorphic to $H^*(M;\Z)\otimes H_{T}^*(\pt;\Z)$   and $i^*:H_{T}^*(M;\Z)\to H_{T}^*(M^{T};\Z)$ is injective, as before.

Part (C) follows by using Proposition~\ref{prop:T on M4} to assemble the equivariant cohomology of the
one skeleton.

To verify Part (D), we use Lemma~\ref{lem:local normal form} to guarantee that the stabilizer group of any
point has at most one (finite) cyclic factor.  This means that $M$ satisfies condition (2.3b) in \cite[Theorem~2.1]{Franz-Puppe2:2011}. 
Because $H_T^*(M;\Z)$ is free over $H^*(BT;\Z)$,
the machinery in \cite{Franz-Puppe2:2011} now implies that the Atiyah-Bredon sequence is exact over $\Z$.

Finally, Franz and Puppe  \cite{Franz-Puppe2:2011} proved that (E) is an immediate consequence of (D).
\end{proof}

\noindent We conclude with an example of a complexity one space and indicate some of 
the computations our work allows.

\begin{figure}[!ht]
\centering
\includegraphics[width=0.95\textwidth]{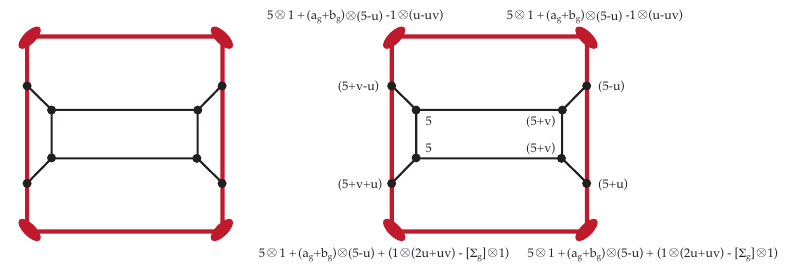} 
    \caption{
    On the left, the x-ray for a complexity one space for a $T^2$ acting on $M^6$. 
    The red fat vertices correspond to  genus $g$ surfaces fixed by $T$.  The
    black vertices correspond to isolated fixed points.  The red edges correspond to
    four-manifolds fixed by a circle, and the black edges correspond to $2$-spheres fixed
    by a circle. 
    This manifold has Betti numbers $\beta_0=\beta_6=1$, $\beta_1=\beta_5=g$, 
    $\beta_2=\beta_4=7$ and $\beta_3=2g$.
    On the right, a collection
of classes in $H^*_{T}(F)=H^*(F)[u,v]$ for each fixed component $F$.
These classes satisfy the requirements in Corollary~\ref{thm:complexity one},
so they are the restrictions to the fixed sets of a global class in
$H_{T}^*(M;\Q)$. }  
\label{fig:complexity one example}
\end{figure}

  \appendix

\section{Computing the intersection numbers  of gradient spheres as a consequence of the ABBV relation}

In principle, the ABBV relation applies when working with coefficients in a field.  However,  we
may still use ABBV to deduce information about integral classes and their cup products as follows.
For a four-manifold, the intersection form is an invariant of integral homology.  In the presence of a group action,
it is the shadow of an equivariant invariant for invariant submanifolds. We may compute the self-intersection 
of an embedded invariant surface by
using ABBV on an integral class (which is, after all, also rational), finding the equivariant invariant
which is a priori rational, but for an integral class is actually an integer.

Let $S$ be an embedded invariant surface in a closed Hamiltonian $S^1$-space $M$ of dimension four and $[S]$ its class in $H_{2}(M;\Z)$.
Then  
 \begin{eqnarray} \label{eq:inter2}
 [S] \cdot [S]&:=&\pi^{!}(\PD([S]) \cup \PD([S]))\\ \nonumber
                            &=&   \pi^{!} \Big({I^{*}}|_{H_{S^1}^{0}(M)}({\iota_{S}}^{!}(1) \cup {\iota_{S}}^{!}(1)) \Big)\\ \nonumber
                   &=& {I^{*}}|_{H_{S^1}^{0}(\pt)} \Big( \pi^{!}({\iota_{S}}^{!}(1) \cup {\iota_{S}}^{!}(1)) \Big) \\ \nonumber
                   &=& \pi^{!} \Big({\iota_{S}}^{!}(1) \cup {\iota_{S}}^{!}(1) \Big) \\ \nonumber
                   &=& \sum_{F \subset M^{S^1}}{\pi|_{F}}^{!} \Big(\frac{({\iota_{S}}^{!}(1)|_{F})^2 }{e_{S^1} (\nu(F \subset M))} \Big)\\ \nonumber
                 \end{eqnarray}
where $\PD$ stands for Poincar\'e dual, 
 the cup product $\cup$  in the first row is in standard cohomology and the following are in equivariant cohomology, the pushforward $\pi^{!}$ in the first two rows are in standard cohomology and the following are in equivariant cohomology; $I^*$ is the map induced by the fiber inclusion $I \colon N \to {{(N \times ES^1)}/{S^1}}$ for $N=M$ and $N=\pt$. The third equality is
 since, by definition of the equivariant pushforward map, the diagram of morphisms  
  \begin{equation} \label{dic}
\begin{CD}
H_{S^1}^{*}(M) @> \pi^{!} >> H_{S^1}^{*-4+0}(\pt)\\
@V I^{*} VV @V I^{*} VV\\
H^{*}(M) @> \pi^{!} >> H^{*-4+0}(\pt)
\end{CD}
\end{equation}
is commutative. 
  The second and fourth equalities are  since the restriction of 
 $ {I^{*}}$ to $H_{S^1}^{0}(N)$ is an isomorphism onto $H^{0}(N)$: one-to-one since the intersection $\ker I^{*} \cap H_{S^1}^{0}(N)=\langle \pi^{*}(u) \rangle \cap H_{S^1}^{0}(N)=\{0\}$; onto by Ginzburg's theorem \ref{thm:module} over $\Q$ and by item (C) in Theorem \ref{thm:S1onM4} over $\Z$.
The last equality is by ABBV \eqref{eq:abbv}. 
 
 \begin{noTitle} \label{grad}
For an $S^1$-invariant $\omega$-compatible structure $J$, the pair $(J,\omega)$  determines an $S^1$-invariant Riemannian metric  $\langle \cdot, \cdot \rangle:=\omega(\cdot,J\cdot)$. We call such a metric {\bf compatible}.
The gradient vector field of the moment map with respect to a compatible metric, characterized by $\langle v,\grad \Phi \rangle=d \Phi(v)$, is
\begin{equation}\label{gradient}
\grad \Phi = -J \xi_M,
\end{equation}
where
 $\xi_M$ is the vector field that generates the $S^1$ action.
The vector fields $\xi_M$ and $J\xi_M$ generate a $\C^\times = (S^1)^\C$ action.
The closure of a non-trivial $\C^\times$ orbit is a sphere, called a {\bf gradient sphere}. On a gradient sphere, $S^1$ acts by rotation with two fixed points at the north and south poles; all other points on the sphere have the same stabilizer. 
We say that a gradient sphere is {\bf free} if its stabilizer is trivial; otherwise it is {\bf non-free}.
 In a compatible metric on $S^1 \acts (M^4,\omega)$, every non-free gradient sphere is a {\bf $\Z_\ell$-sphere} for some $\ell>1$, i.e.,
a connected component of the closure of the set of points in $M$ whose stabilizer is equal to the cyclic subgroup of $S^1$ of order $\ell$, and every $\Z_{\ell}$-sphere is a gradient sphere \cite[Lemma 3.5]{Karshon:1999}.
Note that there is an isotropy weight $\ell$ at the south pole of the $\Z_\ell$ sphere,
and weight 
$-\ell$ at the north pole.
\end{noTitle}

Let $S$ be a gradient sphere with respect to a compatible metric in a Hamiltonian $S^1 \acts (M^4,\omega)$, and $p$ and $q$ its north pole and south pole.
Assume that 
\begin{itemize}
\item if $p$ is an isolated fixed point, then there is a gradient sphere $S_{+}$ such that $$
p \mbox{ is the } \begin{cases}
\text{south pole of }S_{+} & \text{ if }p \text{ is interior},\\
\text{north pole of }S_{+} &\text{ if }p \text{ is maximal}
\end{cases};$$
\item if $q$ is 
an isolated fixed point then there is a gradient sphere $S_{-}$ such that 
$$
q \mbox{ is the } \begin{cases}
\text{north pole of }S_{-} &\text{ if }q \text{ is interior},\\
\text{south pole of }S_{-} & \text{ if }q \text{ is minimal}
\end{cases}.$$
\end{itemize}
If $p$ is on the maximal surface $\Sigma_{\max}$ set $S_{+}=\Sigma_{\max}$, 
and if $q$ is on the minimal surface $\Sigma_{\min}$ set $S_{-}=\Sigma_{\min}$.
Note that if $S$ is a gradient sphere whose image under the moment map is an edge in the decorated graph associated to the Hamiltonian $S^1$-space in \cite{Karshon:1999}, then the above assumptions hold.

Denote by $\ell$ the order of the stabilizer of $S$; set $\ell=1$ if $S$ is a free gradient sphere. If $S_{+}$ is a gradient sphere denote by $\ell_{+}$ the order of its stabilizer; set $\ell_{+}=1$ if $S_{+}$ is free. If $S_{+}$ is a fixed surface set $\ell_{+}=0$. Similarly denote 
$\ell_{-}$.
The normal bundle of $S$ can be viewed as an equivariant complex line bundle over $S^2$ \cite[Corollary A.6]{Karshon:1999}; the circle action is linear on the fibers over the north and south poles with weights
\begin{equation} \label{eq:weights}
 (-1)^{\delta_{p=\max}}\ell_{+} \text{ at }p; \,\, (-1)^{\delta_{q\neq\min}}\ell_{-} \text{ at }q.
 \end{equation}
If $p$ is an isolated fixed point then $TM_{p}$ splits as the normal bundle to $S$ and the normal bundle to $S_{+}$ with weights $(-1)^{\delta_{p=\max}}\ell_{+}$ and $-\ell$. Similarly the weights at $q$ are $\ell$ and $(-1)^{\delta_{q\neq\min}}\ell_{-}$ if $q$ is an isolated fixed point. If  $p$ ($q$) is on a fixed surface $\Sigma_{*}$ ($*=\max$ for $p$ and $\min$ for $q$) then the weights are $0,-1$ (respectively $1,0$), as explained in \S \ref{weights}. 
 
 For a connected component $F$ of $M^{S^1}$ we have 
 $${{\iota_{S}}^{!}(1)|_{F}}:=\iota^{*}_{F  \hookrightarrow M} \circ \iota^{!}_{S \hookrightarrow M}(1_{S}^{S^1})=\iota^{!}_{F \cap S \hookrightarrow F} \circ \iota^{*}_{F \cap S \hookrightarrow S}(1_{S}^{S^1}).$$
 In particular, if $S$ and $F$ do not intersect then ${{\iota_{S}}^{!}(1)|_{F}}=0$.
If $F=\{p\}$ then the restriction ${{\iota_{S}}^{!}(1)|_{F}}=\iota^{!}_{\{p\} \hookrightarrow \{p\}} \circ \iota^{*}_{\{p\} \hookrightarrow S}(1_{S}^{S^1})$ is the equivariant Euler class of the complex one-dimensional normal $S^1$-representation of $S$ at $p$, hence, by \eqref{eq:weights} and  \eqref{eq euler0} (with $n=1$),
 ${\iota_{S}}^{!}(1)|_{F}=-(-1)^{\delta_{p=\max}} \ell_{+} t$. Similarly, if $F=\{q\}$ then ${\iota_{S}}^{!}(1)|_{F}=-(-1)^{ \delta_{q\neq\min}}\ell_{-} t$. 
 If $F$ is a fixed surface and $F$ intersects $S$ at $p$ (at $q$) then $\ell_{+}=0$ ($\ell_{-}=0$) and so is  ${\iota_{S}}^{!}(1)|_{F}$.

Therefore, 
 in case $p$ and $q$ are isolated fixed points, 
\eqref{eq:inter2} and \eqref{eq:eu-11} imply that \begin{equation} \label{first}
[S] \cdot [S]=\begin{cases}
\frac{\ell_{+}^2 t^2}{-\ell_{+} \cdot  \ell t^2}+ \frac{\ell_{-}^2 t^2}{- \ell \cdot \ell_{-} t^2}
                   =\frac{-\ell_{+}-\ell_{-}}{\ell} & \text{ if $p$ and $q$ are interior }\\
                  \frac{\ell_{+}-\ell_{-}}{\ell}   & \text{ if $p$ is $\max$ and $q$ is interior} \\
          \frac{-\ell_{+}+\ell_{-}}{\ell}           & \text{ if $p$ is interior and $q$ is $\min$}\\
                \frac{\ell_{+}+\ell_{-}}{\ell}     & \text{ if  $p$ is $\max$ and $q$ is $\min$}\\
\end{cases}.
\end{equation} 
The first case in \eqref{first} is proven in \cite[Lemma 5.2]{Karshon:1999} by a different proof (not using ABBV).
Moreover,
\begin{equation}
[S] \cdot [S]=\begin{cases}
\frac{-\ell_{-}}{\ell}& \text{ if $p \in \Sigma_{\max}$ and $q$ is isolated interior }\\
 \frac{\ell_{-}}{\ell} & \text{ if $p \in \Sigma_{\max}$ and $q$ is isolated $\min$}\\
       \frac{-\ell_{+}}{\ell}              & \text{ if $p$ is isolated interior and $q \in \Sigma_{\min}$} \\
           \frac{\ell_{+}}{\ell}              & \text{ if $p$ is isolated $\max$ and $q \in \Sigma_{\min}$}\\
        0             & \text{ if  $p \in \Sigma_{\max}$ and $q \in \Sigma_{\min}$}\\
\end{cases}.
\end{equation}

We note that, by a similar argument, if $S$ and $S'$ are gradient spheres with respect to a compatible metric, and $S \neq S'$, then $[S] \cdot [S']$ is zero if $S \cap S' =\emptyset$ or $S \cap S'$ is a point on a fixed surface and one if $S \cap S'$ is an isolated fixed point.

\end{document}